\documentclass[reqno,10pt]{amsart}

\usepackage{filecontents}\usepackage[margin=3cm]{geometry}
\usepackage{amssymb,amsfonts}
\usepackage{mathtools}
\usepackage{nicematrix}
\usepackage{enumerate}
\usepackage{mathrsfs}

\usepackage{thm-restate}
\usepackage{mathtools}
\usepackage{xcolor}
\usepackage{hyperref}
\usepackage{cleveref}

\usepackage{blkarray}
\usepackage{amsthm}
\theoremstyle{definition}
\newtheorem{thm}{Theorem}[section]
\theoremstyle{definition}

\newtheorem{cor}[thm]{Corollary}
\newtheorem{prop}[thm]{Proposition}
\newtheorem{lem}[thm]{Lemma}

\newtheorem{assume}[thm]{Assumption}

\newtheorem{defn}[thm]{Definition}

\newtheorem{exmp}[thm]{Example}
\newtheorem{nonexmp}[thm]{Non-example}
\newtheorem{exmps}[thm]{Examples}
\newtheorem{notn}[thm]{Notation}

\newtheorem{rem}{Remark}[section] 
\def\A{{\mathbb A}}
\def\C{{\mathbb C}}
\def\Z{{\mathbb Z}}

\def\Q{{\mathbb Q}}
\def\A{{\mathbb A}}

\def\F{{\mathbb F}}
\def\N{{\mathbb N}}

\def\C{{\mathbb C}}

\DeclareMathOperator{\End}{End}

\DeclareMathOperator{\GL}{GL}

\DeclareMathOperator{\Id}{Id}

\DeclareMathOperator{\cris}{cris}
\DeclareMathOperator{\Hom}{Hom}

\DeclareMathOperator{\Spec}{Spec}

\DeclareMathOperator{\sep}{sep}

\DeclareMathOperator{\Gal}{Gal}
\DeclareMathOperator{\Ker}{Ker}

\newcommand{\tensor} {\otimes}

\def\o{{\mathcal O}}

\def\bbG{{\mathbb G}}

\def\F{{\mathbb F}}
\def\F{{\mathbb F}}

\DeclareMathOperator{\cst}{cst}

\DeclareMathOperator{\Rep}{Rep}
\DeclareMathOperator{\Vector}{Vec}

\DeclareMathOperator{\Sh}{Sh}

\DeclareMathOperator{\ur}{ur}

\DeclareMathOperator{\GSpin}{GSpin}

\DeclareMathOperator{\SO}{SO}

\DeclareMathOperator{\Frob}{Frob}
\DeclareMathOperator{\perf}{perf}

\DeclareMathOperator{\coker}{coker}

\DeclareMathOperator{\For}{For}
\DeclareMathOperator{\Aut}{Aut}

\usepackage{tikz-cd}
\usepackage[all,cmtip]{xy}

\hypersetup{colorlinks=true,
linkcolor=blue
}

\begin{document}
\title{ Local monodromy of unit root F-isocrystals from Shimura varieties}
\author{Tejasi Bhatnagar}
\address{Department of Mathematics, The Ohio State University.}
\email{tejasi.bhatnagar@gmail.com}
\keywords{Abelian varieties, Galois representaions, F-isocrystals, monodromy}

\begin{abstract}
We generalise a local $p$-adic monodromy theorem of Igusa that studies the monodromy representation associated to the universal elliptic curve around a supersingular point of the modular curve. Let $X$ be a smooth quasi-projective variety over $\F_p$. We set up and prove the analog of Igusa's theorem for overconvergent F-isocrystals on $X$ such that the action of the Frobenius is algebraic, $p$-plain, and semisimple at the closed points of $X$. We study the local monodromy of their unit root sub-objects around a point in $X$ with isoclinic Newton slopes. In particular, our result generalises Igusa's Theorem to overconvergent F-isocrystals arising from Shimura varieties, unconditionally for Shimura varieties of abelian type and conditional on Frobenius semisimplicity for exceptional Shimura varieties where this property is not known yet. In the particular case of Siegel Shimura varieties, we also prove an analogous result for the local monodromy of a point in the boundary of its compactification. As a consequence of our results, we prove a finiteness result for the reduction of the Hecke orbit of abelian varieties over a local field of equicharacteristic.
\end{abstract}
\maketitle


\section{Introduction}
Throughout this paper, we denote by $K$ the local function field $\F_q((t))$ in characteristic $p$. We let $\mathcal{O}_K$ be its valuation ring and $k$ its residue field.  In \cite{igusa}, Igusa studied the local monodromy representation of the universal elliptic curve around a supersingular point of the modular curve and showed that, as we attach (coordinates of) its $p$-power torsion to the base field $K$, we get totally ramified extensions. More precisely, this yields the following result.

\begin{thm}[Igusa]\label{thm:igusa}
Let $E/K$ be an ordinary elliptic curve with supersingular reduction. Let $\rho_E$ be the Galois representation associated to the $p$-power torsion points of $E$. Then the image of inertia is finite index in the image of $\rho_E$. 
\end{thm}
Igusa's theorem shows that the $p$-adic Galois representation associated to the $p$-power torsion points of the universal elliptic curve around a supersingular point is highly ramified, much in contrast to the unramified $l$-adic Galois representation coming from the $l$-power torsion points for $l\neq p$. More generally, we can study similar questions for F-isocrystals coming from a compatible system of $l$-adic and $p$-adic coefficient objects on a variety $X$ defined over $\F_p$. At primes $l\neq p$, these are lisse $\Q_l$-sheaves and at $l=p$, these are \textit{overconvergent} F-isocrystals over a variety defined over $\F_p$. They are called \textit{companions} in the sense that they exhibit various properties that arise from the cohomology of a smooth proper variety over $X$. That is, the action of the Frobenius on the fibres of the coefficient objects at closed points $x\in X$ is algebraic and behaves compatibly at all primes. A first example of our interest of such a compatible system is the crystalline cohomology of an elliptic curve over $\F_q$ which is a companion to its $l$-adic \'etale cohomology at primes $l\neq p$.

In fact, Deligne in Weil II \cite{DeligneWeil2} predicted that any irreducible $l$-adic coefficient object (up to a twist) admits companions at every prime. A remarkable application of Deligne's conjecture, now established by works of many people, (see \cite{kedlayacompanion} for a precise history) is studying local systems with \textit{geometric origin} (see for instance \cite{EGintegrality}). An important class of examples studied extensively are canonical $l$-adic local systems and $p$-adic F-isocrystals on mod $p$ fibers of Shimura varieties. We do not undertake the study of compatible coefficient objects and their origins; rather we \textit{use} the properties of such overconvergent F-isocrystals to generalize \Cref{thm:igusa}.

Coming back to Igusa's result, \Cref{thm:igusa} studies the Galois representation coming from the \'etale quotient of the $p$-divisible group of an elliptic curve. In the same spirit, the purpose of this paper is to generalize \Cref{thm:igusa} to study the local monodromy representation associated to the \textit{unit root} sub- objects of overconvergent F-isocrystals over an open punctured disc coming from such a compatible system. 

\subsection{Main results}\label{mainresult}
We now describe our set-up and main results precisely. 
Let $X$ be a smooth quasi-projective variety over $\F_p$. Let $\mathcal{M}^{\dagger}$ be an  overconvergent F-isocrystal on $X$. We write $\mathcal{M}^{\dagger}_x$ for its pullback to a closed point $x\in X$.

\vspace{2mm}
We assume the following for $\mathcal{M}^{\dagger}$.
\begin{enumerate}
\item For all closed points $x\in X$, the characteristic polynomial of the Frobenius on $\mathcal{M}^{\dagger}_x$ has coefficients in a number field, that is, $\mathcal{M}^{\dagger}$ is \textit{algebraic}. 
\vspace{2mm}
\item  At all closed points $x\in X$, if all the roots of the characteristic polynomial of the Frobenius are $l$-adic units for $l\neq p$, that is,  $\mathcal{M}^{\dagger}$ is \textit{$p$-plain}.
\vspace{2mm}
\item The action of the Frobenius on closed points $x\in X$ is semisimple.
\end{enumerate}
\vspace{2mm}
The underlying convergent F-isocrystal $\mathcal{M}$ defines a Newton stratification on $X$. Let $Z$ be a Newton stratum of $X$ such that $\mathcal{M}^{\dagger}|_{Z}$ has a non-trivial unit root sub-object $\mathbb{U}$. Denote by $\overline{Z}$ the closure of $Z$ in $X$. Write $\eta$ to be the generic point of $Z$. Let $x$ be a $k$ point of $\overline{Z}$ which is a specialisation of $\eta$, so that there exists a map $\Spec \o_K=:S\rightarrow \overline{Z}$ sending the closed point of $S$ to $x$ and $\Spec K$ to the generic point of $Z$. Pulling back $\mathbb{U}$ to $\eta$ yields a local Galois representation: $$\rho_x:\Gal(K^{\sep}/K)\rightarrow \GL(V)$$
where $V$ is the space of Frobenius invariant-sections of $\mathbb{U}$ base changed to a suitable separable closure of the base. The representation $\rho_x$ is our main object of study.
\begin{defn}
We say a closed point of $x\in X$ is isoclinic if the Newton polygon of $\mathcal{M}^{\dagger}_x$ has one slope with a repeated value. 
\end{defn}
We now state our main result.
\begin{thm}\label{thm:mainthm}
Let $\mathcal{M}^{\dagger},Z, \overline{Z}, x,$ and $\rho_x$ be as above. Let $x$ be an interior point of $\overline{Z}(\F_q)$. Suppose the F-isocrystal $\mathcal{M}^{\dagger}_x$ of $x$ is isoclinic, then the image of the inertia is finite index in the image of $\rho_x$. 

\end{thm}
\subsection{Application to Shimura varieties.}
For convenience, we follow the notation of \cite{Patrikis}. Let $(G,\mathbb{X})$ be a Shimura datum, $K_0$ a neat compact open subgroup of $G(\A_f)$ and $\Sh_{K_0}(G,\mathbb{X})$ the associated Shimura variety defined over the number field $E(G,\mathbb{X})$. As we vary the level $\mathbb{K}\subset K_0$ over compact open subgroups of $G(\A_f)$, we get a morphism of the associated Shimura varieties $\Sh_{\mathbb{K}}(G,\mathbb{X})\rightarrow \Sh_{K_0}(G,\mathbb{X})$ over $E(G,\mathbb{X})$ and we denote $\Sh:=\lim_\mathbb{K}\Sh_\mathbb{K}(G,\mathbb{X})$. We choose $s\in \Sh(\C)$ in the inverse system to give a compatible choice of base-points at each level. Let $\Sh_{\mathbb{K},s}$ be a geometrically connected component of $\Sh_\mathbb{K}(G,\mathbb{X})$ containing $s$ defined over an extension $E_{\mathbb{K},s}$ of $E(G,\mathbb{X})$. As in \cite{Patrikis}, we assume that $G$ is such that the center $Z_G(\Q)$ is a discrete subgroup of $G(\A_f)$. In this case, for any normal compact open subgroup $\mathbb{K}\subset K_0$ we get a Galois cover $\Sh_{\mathbb{K}}(G, \mathbb{X})\rightarrow \Sh_{K_0}(G, \mathbb{X})$. For each $l$, we get canonical $l$-adic local systems:
$$\rho_l:\pi_1(\Sh_{K_0,s},s)\rightarrow G(\Q_l)$$

When $(G,\mathbb{X})$ is of abelian-type, we know that there exist canonical integral models for $\Sh_K(G,\mathbb{X})$ by work of Kisin in \cite{intmod}. More recently, Bakker-Shankar-Tsimerman \cite{intmodelexceptional} construct canonical models for all Shimura varieties, including when $G$ is of exceptional type. Let $N$ be a suitable integer and $\mathscr{S}_{K_0,s}$ over $\o_{E_{K_0,s}}[1/N]$ be the canonical integral model of $\Sh_{K_0,s}$. Works of Kisin \cite{abeliantypemodp},\cite{abeliantypemodp} for abelian type Shimura varieties and recent works of Klevdal-Patrkis \cite{Patrikis-Klevdal}, Huryn-Kedlaya-Klevdal-Patrikis, \cite{Huryn-Patrikis-Klevdal-Kedlaya}, and Patrikis \cite{Patrikis},  for exceptional Shimura varieties show that $\rho_l$ extend to arithmetic local systems $\pi_1(\mathscr{S}_{K_0,s}[1/l])\rightarrow G(\Q_l)$ for all primes $l$. Moreover, at closed points $x$ of $\mathscr{S}_{K_0,s}$, the image the Frobenius $\rho_l(\Frob_x)$  in $G(\Q_l)$ is independent of $l$.

Let $x\in \mathscr{S}_{K_0,s}$ be a closed point with residue characteristic $p$. Our interest is when $l=p$. Let $\mathscr{S}_{k(v)}$ be the special fibre of $\mathscr{S}_{K_0,s}$ at some place $v \mid p$. In this case, there exists an overconvergent $G$-F-isocrystal  associated to the crystalline local system $\rho_p$ (see for instance, \cite{esnault-Groechnig_overconvergent_crystal} Theorem 1.2, and \cite{Huryn-Patrikis-Klevdal-Kedlaya} Theorem 1.1) that extends logarithmically to a suitable toroidal compactification of $\mathscr{S}_{k(v)}$ and has nilpotent residues along the boundary divisor.

Let $V$ be an irreducible representation of $G_{\Q_p}$ and $\mathcal{M}^{\dagger}$ be the associated overconvergent F-isocrystal on $\mathscr{S}$. 
 Let $\mathcal{M}^{\dagger}_x$ be its pullback to $x$.   Then the above results on compatibility of the linearized Frobenius on $\mathcal{M}^{\dagger}_x$ with the canonical $l$-adic local systems ($l\neq p$), 
 allow us to apply \Cref{thm:mainthm} to $\mathcal{M}^{\dagger}$ on $\mathscr{S}_{k(v)}$. Let $Z$ be a Newton stratum as defined in \Cref{mainresult} such that its closure $\overline{Z}$ in $\mathscr{S}_{k(v)}$ intersects the basic stratum. Write $\eta$ for the generic point of $\overline{Z}$ and let $\mathcal{M}^{\dagger}_{\eta}$ be the pullback of $\mathcal{M}^{\dagger}$ to $\eta$. Let $\mathbb{U}$ be any unit root sub-object of $\mathcal{M}^{\dagger}_{\eta}$. For a closed point $x\in \overline{Z}(\F_q)$, we write $\rho_x$ to be the Galois representation associated to $\mathbb{U}$ as in \Cref{mainresult}.
\begin{assume}[Frobenius semisimplicity]\label{assumtionfrobss}
Let $(G,\mathbb{X})$ be a Shimura datum as above. It is expected that the action of the linearised Frobenius on $\mathcal{M}^{\dagger}_x$ is semisimple for all closed points $x$ of $\mathscr{S}_{K_0,s}$. For the following application, we assume that this property holds.
\end{assume}

\begin{cor}\label{cor:goodredSV} Let $(G,\mathbb{X})$ be a Shimura datum such that $Z_G(\Q)$ is a discrete subgroup of $G(\A_f)$.  Let $K_0$ be a neat compact subgroup of $G(\A_f)$. Let $s\in \Sh(\C)$, and $\mathscr{S}_{K_0,s}$ over $\o_{E_{K_0,s}}[1/N]$ be as above. Fix a prime $p\nmid N$. Let $\mathscr{S}_{k(v)}$ be the special fiber of $\mathscr{S}_{K_0,s}$ at some place $v\mid p$. Assume that the basic Newton stratum of $\mathscr{S}_{k(v)}$ is non-empty. Let $V$ be an irreducible representation of $G_{\Q_p}$ and let $\mathcal{M}^{\dagger}$ be the associated F-isocrystal on $\mathscr{S}.$ Let  $Z, \overline{Z}, x, \rho_x$ be as above. Let $x$ be a basic point of $\overline{Z}(\F_q)$. Then assuming \ref{assumtionfrobss} holds, the image of inertia is finite index in the image of $\rho_x$.
\end{cor} 
We note that Frobenius semisimplicity is known in many cases. By \cite{abeliantypemodp} Theorem 0.4 it holds for all closed points when $(G,\mathbb{X})$ is of Hodge type. This, along with Kisin's description of (connected components of) integral models of adjoint abelian type Shimura varieties (\cite{intmod}, Theorem 3.4.10.) implies Frobenius semisimplicity for (all) abelian type Shimura varieties. Therefore, \Cref{cor:goodredSV} is true unconditionally in these cases. The property is also known to hold for $\mu$-ordinary points of all Shimura varieties by \cite{intmodelexceptional} Theorem 1.6. It is indeed expected to hold in general for exceptional Shimura varieties as well, by the forthcoming work of Madapusi and Lee \cite{frobssupcoming}. 
\begin{exmps}
We discuss three examples.
\begin{enumerate}

\item  Let $\mathscr{S}$ be the Siegel Shimura variety over $\F_p$. Take $Z$ to be the ordinary stratum and $x$ a supersingular point of $\mathscr{S}(\F_p)$. Then \Cref{cor:goodredSV} gives a direct generalisation of \Cref{thm:igusa} for the local monodromy representation associated to the $p$-power torsion points of ordinary abelian varieties over $K$ with supersingular reduction. We note that the result as stated, also applies to other Newton strata with $p$-rank greater than $0$.
\item Let $\mathscr{S}$ over $\F_p$ be the orthogonal Shimura variety associated to the  orthogonal group $\SO(n,2)$ with signature $(n,2)$. In particular, when $n\leq 19$, these are moduli spaces of K3 surfaces. We consider the F-isocrystal associated to the standard representation of $\SO(n,2)$. We can take $Z$ to be the ordinary stratum and the associated F-isocrystal of an ordinary $K$-point $\eta$ of $\mathscr{S}$ will have a Newton polygon, with slopes $-1,$ and $1$ of multiplicity $1$ and slope $0$ of multiplicity $n$. In the specific case of K3 surfaces, we are considering the first negative Tate twist of their second crystalline cohomology. The associated Galois representation coming from the unit-root sub-object is the representation:
$$\rho: \Gal(K^{\sep}/K)\rightarrow \SO_n(\Q_p)$$
When $\eta$ specialises to a supersingular point of $\mathscr{S}(\F_q)$, we get that image of inertia is finite index in the image of $\rho$. 

\item Let $G$ be an inner form of the simple, simply connected algebraic group over $\Q$ with Lie algebra of type $E_6$. Let $(G, \mathbb{X})$ be the Shimura datum where $\mathbb{X}$ is the Hermitian symmetric domain of type EIII. As $G$ is an inner form, the Galois action on the Dynkin Diagram is trivial and thus, the conjugacy class of the associated cocharacter $\mu$ is invariant under $\Gal(\overline{\Q}/\Q)$. Therefore, the reflex field of the associated Shimura variety is $\Q$. We choose a large enough prime $p$ so that $G$ is split at $p$ and there exists a good integral model over $\Z_p$ of the exceptional Shimura variety in the sense of \cite{intmodelexceptional} Theorem 1.3. Write $\mathscr{S}$ for its reduction mod $p$. Let $V$ be the 27-dimensional faithful representation of $G$ over $\Q$ corresponding to node 1 of the Dynkin diagram. Let $\mathcal{M}^{\dagger}$ be the associated F-isocrystal on $\mathscr{S}$ associated to $V\otimes \Q_p$. Then the grading induced from the central $\bbG_m$-factor of Levi subgroup associated to $\mu$ gives a weight decomposition on $V\otimes {\overline\Q_p}= V_{-1}\oplus V_0\oplus V_1$ with multiplicity $16, 10$ and $1$ respectively (see \cite{e6eg}, Section 8).  For the generic point $\eta$ of the $\mu$-ordinary stratum, the Newton co-character equals Galois average of the the Hodge co-character over the splitting field of a maximal torus of $G_{\Q_p}$. As $G$ is assumed to be split at the prime $p$, both the co-characters are the same. Hence, we consider the $10$-dimensional Galois representation associated to the unit root F-isocrystal of the slope $0$ part of the Newton polygon. Assuming \ref{assumtionfrobss} holds, when $\eta$ specialises to a point in the basic stratum, we get the conclusion of \Cref{thm:igusa} for the associated Galois representation.
\end{enumerate}
\end{exmps}
\subsection{Abelian varieties with semi-stable reduction.} For Siegel Shimura varieties, we also prove an analog of Igusa's result when $x$ is a point of the boundary of its compactification. We introduce the notation to precisely state the corollary. 

\begin{defn}\label{bbcompdefn}
Let $\mathcal{A}_g$ over $\F_p$ denote the Siegel Shimura variety viewed as the moduli space of abelian varieties in characteristic $p$. Let   $A\in \mathcal{A}_g(K)$ be abelian variety with semi-stable reduction. Then $A$ reduces to the boundary of its Baily Borel compactification. 
As the Baily-Borel boundary components of Siegel Shimura varieties are lower dimensional Siegel Shimura varieties, the reduction will be specified based on the Newton stratum $A$ reduces to. More precisely, if $A$ reduces to the zero dimensional boundary then $A$ then its reduction over $k$ is a torus. Otherwise $A$ has \textit{semi-abelian reduction}. In this case,  if $A$ reduces to the supersingular Newton stratum of the boundary we say that $A$ has \textit{semi-supersingular reduction.}
\end{defn}
The following corollary is a direct analog of \Cref{thm:igusa} for abelian varieties with semi-stable reduction.
\begin{thm}\label{thm:badredabvar}
Let $A$ be an ordinary abelian variety over $K$ with semi-supersingular reduction. Write $\rho_A$ for the Galois representation associated to the $p$-power torsion points of $A(\overline{K})$. Then the image of inertia has finite index in the image of $\rho_A$.
\end{thm}
A direct consequence of our results is the following finiteness result of reduction of Hecke orbits.
\begin{cor}
 Let $A$ be an ordinary abelian variety over $K$ with supersingular reduction (resp. semi-supsersingular reduction). Then the reduction of its Hecke orbit (resp. $p$-power Hecke orbit) is finite.
\end{cor}

\subsection{Previous work and our methods.}
In characteristic $p$, the first result in this direction is due to Igusa mentioned in Theorem \ref{thm:igusa} and subsequent result of Chai \cite{chaimonodromy} proves analogous local monodromy results for one-dimensional $p$-divisible groups. Chai's results study precisely the ramification and the upper breaks of the abelian field extension obtained from the $p$-power torsion
points of its \'etale quotient. In higher dimensions, building on Igusa's result, the author studied the local monodromy of ordinary abelian surface and K3 surfaces with bad reduction (See \cite{myownpaper}). More generally, we work in the set up of orthogonal Shimura varieties  and exploit the geometry of $\GSpin$ Shimura varieties to prove such results in characteristic $p$. Our previous paper crucially uses Igusa's result and the explicit geometry of Shimura varieties at hand, which limit our methods to the orthogonal case. Naturally, we next sought to generalize Igusa's result.

Over mixed characteristic,  \Cref{thm:igusa} is known more generally for abelian varieties due to the work of Kisin-Lam-Shankar-Srinivasan in \cite{KLSS}. The main ingredient of their proof is the $D_{\cris}$ functor in mixed characteristic from $p$-adic Hodge theory to study the filtered isocrystal associated to the $p$-adic Galois representation of an abelian variety with supersingular reduction. The $D_{\cris}$-functor acts as a convenient bridge between the Galois representation that we see generically and the fibre of the abelian variety mod $p$, that is supersingular. Furthermore, the authors use their theorem to prove a finiteness result for the reduction of the Hecke orbit of abelian varieties in such a set up. The argument in \cite{KLSS} over $\Q_p$ served as an inspiration for our methods to prove Igusa's result in higher dimensions in characteristic $p$. 

However, as the $D_{\cris}$ functor used in \cite{KLSS} is only a phenomenon in the mixed characteristic, we turn to the theory of F-isocrystals. Along with the local theory of unit root and overconvergent F-isocrystals, our proof makes crucial use of monodromy groups attached to F-isocrystals. In the local set up, while the sub-category of unit root F-isocrystals is neutral Tannakian, we do not necessarily have that for the ambient category of F-isocrystals over $K$. However, our proof works with \textit{finitely generated} sub-categories of F-isocrystals over $K$ and they are ``small enough" to be neutral (see \cite{delignetannakian}, Corollary 6.20). Once we have \textit{local} monodromy groups, the arguments  in \cite{marco_ambrosi} and \cite{marcoparabolicity} for \textit{global} monodromy groups can be adapted to understand the monodromy of the unit root sub-F-isocrystal while keeping track of the fibre over the closed point of its ambient overconvergent F-isocrystal, much in spirit to the role of the $D_{\cris}$-functor in the argument over $\Q_p$. We note that our proof does not use any moduli interpretation of Shimura varieties at hand which the author exploited previously, but rather, the abstract local theory of F-isocrystals. Hence, we write our result purely in the language of F-isocrystals even though finiteness is  truly a consequence of their properties that `come from geometry'. 

To study the local monodromy of abelian varieties with semi-stable reduction, we utilize the study of their $p$-adic uniformization as in our previous paper \cite{myownpaper}. In contrast to our previous explicit method utilizing the geometry of $\GSpin$ Shimura varieties, we use the theory of the associated log F-isocrystal and its filtrations to generalise the result in \cite{myownpaper} to Siegel Shimura varieties. Finally, following \cite{KLSS} Corollary 2.10, we prove a finiteness result for the Hecke orbit of abelian varieties over a local field of equicharacteristic as a direct application of our results. 
\subsection{Acknowledgments}
I thank Ananth Shankar for suggesting possible generalisations of Igusa's result that led to my thesis and this paper. I am grateful to Marco D'Addezio, Jake Huryn, Ruofan Jiang, Kiran Kedlaya, Alice Lin, Stefan Patrikis, and Ananth Shankar for helpful conversations and/or email correspondences answering my questions. Particular thanks to Kiran Kedlaya for pointing me in the right direction with the proof of \Cref{prop:extfullyfaithful}. I am especially thankful to Stefan Patrikis for an insightful remark (Remark \ref{rem:decent}), that helped me write the proof in a more general set-up than in an earlier draft of this paper, and for very helpful conversations on Tannakian categories.  I acknowledge the conference ``Local Systems in Algebraic Geometry'' held at The Ohio State University in May 2024, which also served as an inspiration to generalise my previous results. I am supported by NSF MSPRF 2503371.
\section{Background} 
\subsection{Local theory of F-isocrystals.}
In this section, we fix notations and definitions, along with some helpful results. We follow \cite{ambrus_p_adic_monodromy} Section 2 and \cite{kedlaya_notes}. Let $\mathcal{O}:= W(k)$ be the ring of Witt vectors over $k$ and $L^{\ur}:= W(k)(1/p)$.  We will need to work with the following rings.
\begin{defn}\label{basicrings}
Let $u$ be a parameter reducing to $t$ mod $p$. 
\begin{enumerate}
\item Let $\Gamma_+$ denote the power series ring that we view as functions over the formal open disc of radius one:
 $$\Gamma_+ = \bigg\{\sum_{i\in \N} x_iu^i\mid x_i\in \mathcal{O}\bigg\}$$
    \item We denote $\Gamma$ to be the $p$ adic completion of $\Gamma_+(1/t)$ given by the sub-ring of bidirectional power series:
$$\Gamma = \bigg\{\sum_{i=-\infty}^\infty x_iu^i\mid x_i\in \mathcal{O}; v_p(x_i)\to-\infty\bigg\}$$

 \item We let $\Gamma^{\dagger}$ to be the Laurant series that converge in some annuli of the form $*\leq|u|<1$, explictly given by
  $$\Gamma^{\dagger} = \bigg\{\sum_{i\in \Z} x_iu^i\mid x_i\in \mathcal{O}; \liminf_{i\to-\infty}\frac{v_p(x_i)}{-i}>0\bigg\}$$
 \end{enumerate}
 For the purpose of this paper, it will be sufficient to invert $p$ and work with $\mathcal{E}:=\Gamma(p^{-1})$, $\mathcal{E}_+:=\Gamma_+(p^{-1})$, and $\mathcal{E}^{\dagger}:=\Gamma^{\dagger}(p^{-1})$. 
 We note that $\Gamma$ and $\Gamma^\dagger$ are discrete valuation rings with residue field $K$ and hence $\mathcal{E}$ and $\mathcal{E}^{\dagger}$ are their fraction fields. The ring $\mathcal{E}_+$ is a discrete valuation ring with uniformizer $u$.
\end{defn}

\begin{defn}
 As is often the case with local results in this area, we will need to work with the Robba ring $\mathcal{R}$, consisting of the union of the rings of rigid analytic functions over $L^{\ur}$ on the annuli of the form $*\leq|u|<1.$ For an explicit expression in terms of power series, see \cite{ambrus_p_adic_monodromy} Definition 2.2.  We note that $\Gamma^\dagger$ is the subring of $\mathcal{R}$ consisting of Laurant series with coefficients in $W(k)$ and hence $\mathcal{E}^{\dagger}\subset \mathcal{R}$. We further define the ``positive Robba ring" $\mathcal{R}_+$ containing $\mathcal{E}_+$ to be the ring $$\mathcal{\mathcal{R}}_+:=\mathcal{R}\cap \big\{\sum_{i\in \N}x_iu^i|x_i\in L^{\ur}\big\}$$
 We view $\mathcal{R}_+$ as the ring of rigid analytic functions on the open unit disc. 
\end{defn}
Let $R$ denote any of the rings $\mathcal{E}_+$, $\mathcal{E}$, $\mathcal{E}^{\dagger},$ $\mathcal{R}$ or $\mathcal{R}_+$ Each of these rings carry a Frobenius map $\sigma(u) = u^q$ and with a derivation given by $d:R\rightarrow \Omega^1_{R}$ where $\Omega^1_{R}$ is the free module over $R$ generated by $du$ such that $u\mapsto du$ and $dr=0$ for all $r\in R$. The action of the Frobenius and the derivation are compatible with each other. Furthermore, we have the map $d\sigma:\Omega^1_R\rightarrow \Omega^1_R$ such that $d\sigma(\sum_ix_iu^i)=\sum_i\sigma(x_i)qu^{iq+q-1}du$. 
\vspace{2mm} 

 We follow the terminology of \cite{kedlaya_notes} Remark 2.11 and define $\textbf{F-Isoc}(\mathcal{O}_K)$, $\textbf{F-Isoc}(K)$, ${\textbf{F-Isoc}^{\dagger}(K)}$, $\textbf{F-Isoc}^{\ddagger}(K)$, and $\textbf{F-Isoc}^{\ddagger}(\o_K)$ the categories of F-isocrystals over the rings $\mathcal{E}_{+}$, $\mathcal{E}$, $\mathcal{E}^{\dagger}$, $\mathcal{R}$ and $\mathcal{R}_+$ respectively as follows:

\begin{defn}\label{defns}
 Let $R$ denote any of the above rings. An F-isocrystal over $R$ is a triple $(M,F,\nabla)$ such that: 

\begin{enumerate}
\item The module $M$ is finite projective over $R$ with an additive and $\sigma$-linear map $F:M\rightarrow M$ that is an isomorphism.

\item A connection $\nabla: M\rightarrow M\otimes \Omega^1_R$ on $M$ that is compatible with the action of the Frobenius $F$, that is, $F\cdot \nabla=\nabla\cdot (F\otimes d\sigma)$. Once we choose a bases of $M$, the connection $\nabla= d+N dt$ where $N$ is a matrix with entries in $R$ which is determined by the action of $\nabla$ on the bases of $M$.

\end{enumerate}
We define the morphisms between two objects in this category given by $R$-linear maps that are compatible with the action of Frobenius and the connection and make the appropriate diagram commute.
\end{defn}
For convenience, we may sometimes drop the notation of the triple $(M,F,\nabla)$ and only write $M$ with the Frobenius and the connection being implicit from context.  

We will keep a running example of abelian varieties for context and to explain our interest in these definitions. 
\begin{exmp}
Let $S=\Spec K$ (resp. $\Spec \o_K$). Let $G[p^{\infty}]$ denote a $p$-divisible group over $\mathcal{O}_K$ (resp. $K$). Then using the crystalline Dieudonn\'e functor $\mathbb{D}(\cdot)$ (see \cite{dejong2} 2.3.3 and 2.3.4), we can associate a Dieudonn\'e crystal (as defined in \cite{dejong2} 2.3.2.) to $G[p^{\infty}]$. These are equivalent to the category of Dieudonn\'e modules over $R=\mathcal{E}$ (resp. $\mathcal{E}_+$) defined in \cite{dejong2} 2.3.4. As $1/p\in R$, using \cite{ambrus_p_adic_monodromy} Lemma 2.17, we further drop the Verschiebung and associate an F-isocrystal over $\mathcal{E}_+$ (resp. over $\mathcal{E}$). In particular, we can associate an F-isocrystal over $\mathcal{E}$ to an abelian variety with semi-stable reduction. When $A$ has good reduction we have that $\mathbb{D}(A[p^{\infty}])=\mathbb{D}(\mathscr{A}[p^{\infty}])\otimes_{\mathcal{E}_+}\mathcal{E}$ where $\mathscr{A}$ denotes the N\'eron model of $A$ over $\mathcal{O}_K$.
\end{exmp}

\begin{defn}
Let $M$ be an F-isocrystal in $\textbf{F-Isoc}(K)$. We say $M$ is \textit{overconvergent} if there exists an F-isocrystal $M^{\dagger}\in {\textbf{F-Isoc}^{\dagger}(K)}$ such that
$M = M^{\dagger}\otimes_{\mathcal{E}^{\dagger} }\mathcal{E}$.\end{defn}

\begin{exmp}\label{abvar_overconv_isoc}
We note the following convenient fact from \cite{de_jong} Section 2.5, and \cite{fabien_overconvergent} Remark 3.17: Let $A/K$ be an abelian variety with semi-stable reduction. Then the associated F-isocyrstal $\mathbb{D}(A[p^{\infty}])$ over $\mathcal{E}$ is overconvergent.
\end{exmp}
\begin{defn}
We will say an F-isocrystal $M$ over $\mathcal{E}$ (resp. $\mathcal{E}^{\dagger}$) is constant if there exists an F-isocrystal $M_0$ over $L$ such that $M=M_0\otimes_L\mathcal{E}$ (resp. $M=M_0\otimes_L\mathcal{E}^{\dagger}$).
\end{defn}
\subsection{Slope filteration and unit root F-isocrystals} \label{slopefilunitroot} We refer the reader to \cite{kedlaya_notes}, Section $3$ for a discussion on slopes and Newton polygons. Every F-isocrystal $M$ in $\textbf{F-Isoc}(K)$ aquires a slope filtration:
$$0=M_0\subseteq M_1\subseteq M_1\subseteq \dots \subseteq M_r=M$$ such that $M_i/M_{i-1}$ is isoclinic of slope $s_i$ and $s_1<s_2<\dots <s_r$. This determines the Newton polygon of $M$. We can also define the Newton polygon of $M$ by pulling back $M$ to $K^{\perf}$, the perfection of $K$ and then using the Dieudonn\'e-Manin classification over the algebraic closure of $K^{\perf}$. 
\begin{defn}
We say an F-isocrystal $M$ over $\mathcal{E}$ or $\mathcal{E}^{\dagger}$ is \textit{unit-root} if all the slopes of its Newton poylogon are zero. We will use the same terminology for Diedonn\'e modules over $W(k)$. 
\end{defn}

Unit root F-isocrystals play the role of \'etale local systems and hence provide us with a bridge between Galois representations coming from abelian varieties and the theory of F-isocrystals. We will need the following result due to Tsuzuki, see \cite{tsuzuki_unit_root} Theorem 4.2.6.
\begin{thm}\label{thm:tsuzuki_unit_root}
The category of unit-root objects in ${\textbf{F-Isoc}(K)}$ is equivalent to the category of continuous representation of the Galois group $G_K:=\Gal(K^{\sep}/K)$ on finite dimensional $\Q_p$-vector spaces. Moreover, the unit-root objects in ${\textbf{F-Isoc}^{\dagger}(K)}$ correspond to  $G_K$-representations with finite image of the inertia subgroup $I_K$.
\end{thm}
\begin{nonexmp}
Let $A/K$ be an ordinary abelian variety. Then the maximal \'etale quotient of $A[p^{\infty}]$ corresponds to a unit-root F-isocrystal over $\mathcal{E}$. Although the F-isocrystal of $A$ over $\mathcal{E}$ is overconvergent, the unit-root sub-object associated to the $p$-power torsion points is \textit{not}.
\end{nonexmp}

\subsection{Logarithmic structures}\label{2.3} For this section, we consider the logarithmic scheme $S = \Spec \o_K$ with the canonical log structure associated to the closed point $t=0$. 

We let $R=\mathcal{E}_+$ or $\mathcal{R}_+$. Write  $\Omega_{R,\log}^{1}:= R \cdot d\log u=R \frac{du}{u}$ for the free module over $R$ generated by $d\log u$ and the logarithmic derivation $d_{\log}:R\rightarrow \Omega_{R,\log}^{1}$ given by $u\frac{d}{du}$. Moreover, we have the map $$d\sigma_{\log}:\Omega_{R,\log}^{1
}\rightarrow \Omega_{R,\log}^{1
} :\sum_i c_iu^i \frac{du}{u}\mapsto \sum_i c_iu^{qi}q\frac{du}{u}.$$
\vspace{-1mm}
\begin{defn}
We define the category of log-isocrystals $\textbf{F-Isoc}_{\log}(S)$ (resp. $\textbf{F-Isoc}_{\log}^{\ddagger}(S)$) consisting of objects that are triples $(M,F,\nabla_{\log})$ over $R=\mathcal{E}_+$ (resp. $\mathcal{R}_+$) such that 
\begin{enumerate}
\item $M$ is a finite projective module over $R$ with the $\sigma$-linear map $F$. \item A logarithmic connection $\nabla_{\log}: M\rightarrow M\otimes\Omega_{R,\log}^{1}$ that is compatible with the action of the Frobenius. That is $F\cdot\nabla_{\log}=\nabla_{\log}\cdot (F\otimes d\sigma_{\log})$.
\end{enumerate}
A morphism of two objects is an $R$-linear map between the two finitely generated modules that is compatible with the action of the Frobenius and the logarithmic connection. We will denote the set of homomorphisms between two objects in this category by $\Hom_{F,\nabla_{\log}}(M_1,M_2)$.
 \end{defn}
 \begin{exmp}
 Let $A/K$ be an abelian varieties with semi-stable reduction. Then the Dieudonn\'e crystal of $A[p^{\infty}]$ extends to a log-Dieudonn\'e on $S$ as defined in \cite{kato_trehan} 4.1. We can equip the Dieudonn\'e crystal naturally with the structure of a log-F-isocrystal over $\mathcal{E}_+$ (see for instance \cite{keerthitor}, 1.4.5). In \Cref{raynaud}, we elaborate more on this. 
\end{exmp}
 
Let $M$ be an object in $\textbf{F-Isoc}_{\log}(S)$. Pulling back $M$ to the divisor $t=0$ yields a finite free module $M_0$ over $L^{\ur}$ along with a semi-linear Frobenius map $F_0$. In addition, any integrable logarithmic connection $\nabla_{\log}$ induces an $L^{\ur}$-linear residue map $N_0: M_0\rightarrow M_0$. Explicitly, if $\nabla_{\log} =d+ N\cdot \frac{dt}{t}$ such that $N$ is matrix with entries in $\mathcal{E}_+$, then $N_0=N|_{t=0}$. Moreover, its compatibility with the Frobenius implies that $F_0\cdot N_0=q N_0\cdot F_0$, forcing $N_0$ to be nilpotent. See \cite{kedlaya_book}, Definition 15.1.1. This functor provides a ``reduction mod $t$" map from the category of log F-isocrystals to log Dieudonn\'e modules over $L^{\ur}$ consisting of triples $(M_0, F_0, N_0)$ satisfying the above properties. In particular we have that:
\begin{prop}[Logarithmic Dwork's trick]\label{prop:redmodtlog}
The category $\textbf{F-Isoc}_{\log}^{\ddagger}(\o_K)$ of log F-isocrystals over $\mathcal{R}_+$ is equivalent to the category of log-Dieudonn\'e modules over $L^{\ur}$. 
\end{prop}
\begin{proof}
Let $(M,\nabla_{\log}, F)$ be any log-isocrystal over $\mathcal{R}_+$. Suppose the connection $\nabla_{\log}=d+N du/u$ with a nilpotent residue $N_0$. Then by \cite{kedlaya_book}, Corollary 15.2.6, we can find a basis of $M$ over $\mathcal{R}_+$ such that the connection becomes $d+N_0 du/u$.  By the compatibility of the Frobenius and the connection, we see that the Frobenius $F$ has entries that are constants. Hence, $M=M_0\otimes \mathcal{R}_+$ where $M_0 = M/tM$ with a action of a Frobenius $F_0=F$ and the nilpotent operator given by $N_0$.
\end{proof}
For $R =\mathcal{R}_+$ and $\mathcal{R}$, we write $R[\log u]$ to be the polynomial ring generated by the symbol $\log u$ with coefficients in $R$. We extend $\sigma$ on $R$ to $R[\log u]$ by $\sigma (\log u) = \log (u^q) :=q\cdot \log u$ and extend derivation $d$ to $d:R[\log u]\rightarrow \Omega^1_{R[\log u],\log}$ by defining $d\log u = \frac{du}{u}$ as above.

The following diagram of faithful functors between the categories defined above will be useful to keep in mind (when $\nabla_{\log}=\nabla,$ a regular connection, this is in \cite{kedlaya_notes}, Remark 2.11): 
\[\begin{tikzcd}
\textbf{F-Isoc}_{\log}(\o_K) \arrow{r}{\For^{\log}}\arrow[swap]{d} & \textbf{F-Isoc}^{\dagger}(K) \arrow[swap]{d}\arrow{r}{\For^{\dagger}}&\textbf{F-Isoc}(K)\\
\textbf{F-Isoc}_{\log}^{\ddagger}(\o_K) \arrow{r} & \textbf{F-Isoc}^{\ddagger}(K)
\end{tikzcd}
\]

We note that by Theorem 1.1 in \cite{kedlayafullfaithfull} the functor $\For^{\dagger}$ is fully faithful. Furthermore, we prove in \Cref{prop:extfullyfaithful} that the functor $\For^{\log}$ is fully faithful.
 
\section{Tannakian categories and the parabolicty theorem}
\subsection{Neutral Tannakian categories} This section is dedicated to checking that the categories we work with are neutral Tannakian. In particular, we define \textit{local} monodromy groups that allow us to prove  \Cref{cor:subquotientoverconverget}, the ultimate goal of this section.
Let $R$ be one of the rings $\mathcal{E}, \mathcal{E}^{\dagger},\mathcal{E}_+$ or $L^{\ur}=W(k)(1/p)$.
\begin{lem}
The category of F-isocrystals over $R$ is abelian.
\end{lem}
\begin{proof}
For $R$ above, denote by $\ker(f)$ for the kernel and $\coker(f)$ for the cokernel of a morphism $f:M\rightarrow N$ between two F-isocrystals. We note that $\ker(f)$ and $\coker(f)$ are invariant under the connection and the Frobenius because of their compatibility with $f$. As $\mathcal{E},\mathcal{E}^{\dagger
}$, and $L$ are fields, we see that $\ker(f)$ and $\coker(f)$ are also finite free. For $R=\mathcal{E}_+$, we see that kernel is finite as $R$ is noetherian and free as $R$ is a principal ideal domain. The cokernel is finite as it is a quotient of $N$. We further note that the only proper ideals of $R$ stable under $d/dt$ are $0$ and itself, that is $R$ is differentially simple. Therefore, by Proposition 1.2.6 in \cite{kedlayaformalstructures}, $\coker(f)$ is free. 
\end{proof}
\begin{notn}\label{notn:tannakian}
For any object $M$ in the category of F-isocrystals over a $R$, we will write $\langle{M}\rangle$ for the full rigid abelian tensor subcategory of F-isocrystals generated by $M$. We note that any object in $\langle{M}\rangle$ is a subquotient of a finite direct sum of tensor constructions of the form $M^{\otimes n_i}\otimes (M^{\vee})^{\otimes m_i}$ which we denote by $T(M)$.
\end{notn}

For the rest of this section, we work over $\mathcal{E}$ and $\mathcal{E}^{\dagger}$.

\begin{lem}\label{lem:tannakian}
The categories $\textbf{F-Isoc}(K)$ and $\textbf{F-Isoc}^{\dagger}(K)$ are Tannakian over $\Q_p$.
\end{lem}
\begin{proof}
We note that both the categories are linear over $\Q_p$ which is the fixed field of the action of $\sigma$ on $\mathcal{E}$ and $\mathcal{E}^{\dagger}$. As $\Q_p$ is a field of characteristic zero, we use Deligne's criteria in \cite{delignetannakian}, Theorem 7.1 to check our claim. Both the categories are rigid, tensor, and abelian. The unit object $\textbf{1}$ in the two categories are the rings $\mathcal{E}$ and $\mathcal{E}^{\dagger}$ respectively and $\End(\textbf{1})=\Q_p$. We check (2) in \cite{delignetannakian}, Theorem 7.1. Let $M$ be an F-isocrystal over $\mathcal{E}$ or $\mathcal{E}^{\dagger}$. As $M$ is finite free, we can compute its categorical dimension by choosing a basis \{$e_i\}_{i=1}^r$ of $M$ and the standard bases $\{e_i^\vee\}_{i=1}^r$ for $M^{\vee}$ with respect to them. Then the categorical dimension is the the map $\textbf{1}\rightarrow\textbf{1}: 1\mapsto\sum_i e_i\otimes e_i^{\vee}\rightarrow \sum_ie_i^{\vee}(e_i)=r\geq 0$. This finishes the proof.
\end{proof}

\begin{lem}\label{lem:neutraltannakian}
Let $M$ and $M^{\dagger}$ be an F-isocrystal over $\mathcal{E}$ and $\mathcal{E}^{\dagger}$ respectively. Then the Tannakian category generated by $M$ (resp. $M^{\dagger}$) in  $\textbf{F-Isoc}(K)$ (resp. $\textbf{F-Isoc}^{\dagger}(K)$) becomes neutral over a finite extension of $\Q_p$.

\end{lem}
\begin{proof}
By \Cref{lem:tannakian}, we know that $\langle{M}\rangle$ and $\langle{M}^{\dagger}\rangle$ are Tannakian over $\Q_p$. Therefore they become neutral over some $\Q_p$-algebra (see \cite{Milne_Deligne_tannakian}, Definition 3.7). As these categories are generated by a single object, we apply \cite{delignetannakian} Corollary 6.20 to see that they are in fact, neutral over a finite field extension of $\Q_p$.
\end{proof}
\subsection{Extension of scalers, Galois descent, and monodromy groups.}
As a consequence of the above Lemma, we will have to work with the categories $\textbf{F-Isoc}(K)$ and $\textbf{F-Isoc}^{\dagger}(K)$ with an $L$-linear structure for a finite extension $L$ of $\Q_p$. We refer to \cite{abe_langlands} Section 1.4.1.
\begin{defn} Let $\mathcal{C}$ be a $\Q_p$-linear abelian category and let $L$ be a finite extension of $\Q_p$. Then define the extension of scalers of $\mathcal{C}$ to $L$, denote as $\mathcal{C}_L$ to consist of objects of the form $M_L:=(M,\rho)$ where $M$ is an object of $\mathcal{C}$ and $\rho: L\rightarrow \End(M)$. The morphisms are morphisms in $\mathcal{C}$ that are compatible with the $L$-structure.

\end{defn}
When $\mathcal{C}$ is abelian, $\mathcal{C}_L$ is also abelian. We note that there exists a natural forgetful functor $\For_L:\mathcal{C}_L\rightarrow \mathcal{C}$ by forgetting the $L$-action on $M$. Going the other way, let $x_1,\cdots,x_d$ be a basis of $L$ over $K$ and let $M\in\mathcal{C}$ and write $\rho_K: M\rightarrow \End(M)$ for the $K$-linear structure on $M$. We define $M\otimes_K L\in \mathcal{A}_L$ with $L$-structure given by right multiplication.  This gives an exact functor $\iota_L:\mathcal{C}\rightarrow \mathcal{C}_L$ which is adjoint to the forgetful functor. 

 For any $\sigma\in \Gal(L/K)$ and an object $M_L=(M,\rho) \in\mathcal{C}$, define $M_L^{\sigma}:=(M, \sigma\circ\rho)$.  Let $M_L\in \mathcal{A}_L$. Then there exists an $M\in \mathcal{C}$ such that $M_L = M\otimes_K L$ if and only if $M_L$ comes with a Galois descent data, that is, isomorphisms $\varphi_{\sigma}: M_L^{\sigma}\rightarrow M_L$ such that $\varphi_{\sigma\tau} = \varphi_\sigma\circ \varphi_\tau^ \sigma$
where $\varphi_\tau^\sigma:(M^{\sigma})^{\tau}\rightarrow M^{\sigma}.$

\vspace{2mm}
We now come back to our set up. Let $M$ and $M^{\dagger}$ be an F-isocrystal over $\mathcal{E}$ and $\mathcal{E}^{\dagger}$ respectively. By \Cref{lem:neutraltannakian} there exists an $L$-linear fibre functor $\omega:\langle{M}\rangle_L\rightarrow \Vector_L$ for some finite extension $L$ of $\Q_p$. The composition of $\omega$ with the forgetful functor $\For_L^{\dagger}:\langle M^{\dagger}\rangle_L\rightarrow \langle M \rangle_L$ gives a fibre functor ${\omega}^{\dagger}:\langle{M}^{\dagger}\rangle_L\rightarrow L$. Write $G:=\Aut^{\otimes}(\omega)$ and $G^{\dagger}=\Aut^{\otimes}(\omega^{\dagger})$ for the associated algebraic monodromy groups defined over $L$ (\cite{scaler_ext} Theorem 3.1.7). The forgetful functor induces a morphism of the associated algebraic groups $f:G\rightarrow G^{\dagger}$. 

\begin{lem}
The morphism $f$ is a closed immersion, that is, $G\subset G^{\dagger}$.
\end{lem}
\begin{proof}
We show that every object of $\langle{M}\rangle_L$ is a subquotient of an object of the form $\For^{\dagger}_L(X)$ where $X\in \langle M^{\dagger}\rangle_L$. Applying \cite{Milne_Deligne_tannakian} Proposition 2.21 (b) that proves the Lemma. Let  $W_L = (W, \rho)\in \langle M \rangle_L$. Then $\For_L(W_L)=W$ is a subquotient of $T(M).$ As $\iota_L$ is exact, we see that $\iota_L(\For_L(W_L))$ is a subquotient of $\iota_L(T(M))$. Using the fact that the natural homomorphism $\iota_L(\For_L(W_L))\rightarrow W_L$ is surjective (\cite{scaler_ext}, Remark 1.3.5), we see that $W_L$ is a subquotient of $\iota_L(T(M))$. As $\For^{\dagger}$ is fully faithful, any $L$ action on $T(M)$ lifts uniquely to an $L$ action on $T(M^{\dagger})$. Therefore, $\iota_L(T(M))=\For^{\dagger}_L (\iota_L (T^{\dagger}M))$ and the claim follows.  
\end{proof}
Let $\langle{M}\rangle_L^{\cst}$ denote the full Tannakian subcategory of $\langle{M}\rangle_L$ of constant $F$-isocrystals and let $G^{\cst}$ denote the associated monodromy group. Similarly define $(G^{\dagger})^{\cst}$ denote the associated monodromy group for $\langle M^{\dagger}\rangle_L^{\cst}$.
\begin{lem}\label{lem:constant}
  The natural morphism $h:G^{\cst}\rightarrow G^{\cst}$ is surjective. Moreover, the groups $G^{\cst}$ and $(G^{\dagger})^{\cst}$ are commutative.
\end{lem}
\begin{proof}
The arguments are same as in \cite{marco_ambrosi} Proposition 2.2.4 and Corollary 2.3.6. By Theorem 1.1 in \cite{kedlayafullfaithfull}, we see that the forgetful functor $\For^{\dagger}$ restricted to $\langle{M^{\dagger}}\rangle^{\cst}$ is fully faithful.  Let $N^{\dagger}$ be a constant F-isocrystal in $\langle{M^{\dagger}}\rangle^{\cst}$. Then every subobject $W$ of $\For^{\dagger}(N^{\dagger})$ is constant and in particular overconvergent. Since $\For^{\dagger}$ is fully faithful, it comes from a constant subobject of $W^{\dagger}\subset N^{\dagger}$. By \cite{Milne_Deligne_tannakian} Proposition 2.21 (a), we see that $h$ is surjective. As in \cite{marco_monodromy} Theorem A.2.2 , we note that $\textbf{F-Isoc}(K)^{\cst}\simeq \Rep_{\Q_p}(\Z)$ as the data of a constant F-isocrystal over $K$ is an $L^{\ur}$-vector space with a semi-linear action of the frobenius, which the same as a $\Q_p$-vector space with a $\Q_p$-linear Frobenius endomorphism. Therefore the Tannakian group of $\textbf{F-Isoc}(K)^{\cst}$ is the pro-finite completion of $\Z$ and $G^{\cst}$ and $(G^{\dagger})^{\cst}$ are quotients of the group. In particular, they are commutative. 
\end{proof}

\subsection{The parabolicity theorem and an application}

The goal of this section is to prove \Cref{cor:subquotientoverconverget}. The global version of this result is studied in \cite{marco_ambrosi} Section $2$ and \cite{marcoparabolicity} Section $4$ using monodromy groups. We are indebted to these works and we make use of these arguments for the local analogue as well.

By the discussion in \Cref{slopefilunitroot}, we can define a $\otimes$-exact filtration $\omega^{\dagger}_n:\langle M^{\dagger}\rangle\rightarrow \Vector_L$ by defining $\omega_n(M^{\dagger})=\omega(M_n)$ where $0\subset M_0\subseteq M_1\subseteq\dots \subseteq M$ is the slope filtration of $M$ over $\mathcal{E}$. Let $P$ denote the subgroup of $G^{\dagger}=\Aut^{\tensor}(\omega^{\dagger})$ that preserve the filtration $\omega^{\dagger}_{\bullet}$ (\cite{saavedra} Proposition 2.2.5). Let $r$ be the lcm of the denominators of slopes of $M$ and $\mathbb{G}_m^{1/r}$ be the torus with character group $1/r\Z$. Then replacing $L$ by possibly a finite extension, we get a Newton cocharacter $\mu:\mathbb{G}_m^{1/r}\rightarrow G^{\dagger}$ acting on $G^{\dagger}$ by conjugation (\cite{saavedra}, Section 2.1.1).

In the particular case when ${M^{\dagger}}$ is semisimple, then (the identity component of) $G^{\dagger}$ is reductive and, $P$ is a parabolic subgroup of $G$. Notice that as $G$ preserves the filtration $\omega^{\dagger}_{\bullet}$, we see that $G\subseteq P$. 

\begin{rem}
We note that D'Addezio's remarkable results in \cite{marcoparabolicity}, proving Crew's parabolocity conjecture, show that in fact, $G=P$. Although the set up in \cite{marcoparabolicity} is purely for F-isocrystals on a variety over $\F_q$, the proof of this conjecture relies on showing that a certain ``minimal slope property'' holds for F-isocrystals on a variety over $\F_q$. As a first step, D'Addezio notices that this property holds for F-isocrystals over $\mathcal{E}$ and then the rest of the work is to ``algebraise the argument'' for the property to hold globally. This fact, along with the proof of Proposition 4.3.2 in \cite{marcoparabolicity} proves that $G=P$. For the sake of completeness of this article and the reader's convenience, we write the important results below even though the proof from the paper remains unchanged.
\end{rem}
We refer the reader to \cite{marcoparabolicity}, Section $4$.
\begin{defn}[\cite{marcoparabolicity}, Definition 4.1.1]
 Let $M$ be an overconvergent F-isocrystal over $\mathcal{E}$ and let $N$ be a sub F-isocrystal of $M$. We define the $\dagger$-hull $\overline{N}$ of $N$ to be the smallest overconvergent sub F-isocrystal of $M$ containing $N$.
\end{defn}
In \cite{marcoparabolicity}, Construction 4.2.3 constructs the $\dagger$-hull of a sub F-isocrystal $N$ in an ambient overconvergent F-isocrystal $M$. The definition and this construction make sense for the $L$-linear category $\langle{M}\rangle_L$ as well. Moreover, Proposition 4.2.2 proves that the first slope piece of $N$ is the same as that of its $\dagger$-hull $\overline{N}$. As the  slope filteration and the reverse slope filtration of F-isocrystals over a given base stays invariant under endomorphisms, the proof of Proposition 4.2.2 goes through when we work with the $L$-linear category as well.

\begin{thm}[\cite{marcoparabolicity}, Proposition 4.3.2]\label{thm:parabolic} We have $P=G\subset G^{\dagger}$. In particular, when $G^{\dagger}$ is reductive then $G$ is a parabolic subgroup of $G^{\dagger}$. 
\end{thm}
\begin{proof}
By Chevalley's theorem there exists an F-isocrystal $N^{\dagger}$ in $\langle M^{\dagger}\rangle_L$ and a rank one sub F-isocrystal $L\subset N$ such that $H$ is the stabiliser of the line $\omega(L)\subset \omega(N):=V$. We now show that $P$ stabilises $\omega(L)$. Let $\overline{L}$ be the $\dagger$-hull of $L\subset N$. Let $s$ be the slope of $L$ and write $V^{\leq s}\subseteq V$ for the subspace of $V$ of slope $\leq s$. By Proposition 4.2.2 in \cite{marcoparabolicity}, we know that the first slope pieces $L=L_1 = \overline{L}_1$ which imples $\omega(L)=\omega(\overline{L})\cap V^{\leq s}$. As $\overline{L}\subset N$ is overconvergent, the monodromy group $G^{\dagger}$ and hence, $P$ stabilises $\omega(\overline{L})$. Moreover $P$ also stabilises $V^{\leq s}$ as $N^{\dagger}\in \langle{M^{\dagger}}\rangle.$ This shows that $P$ stabilises $\omega(L)=\omega(\overline{L})\cap V^{\leq s}$. 
\end{proof}

\begin{cor}\label{cor:constovercon}(\cite{marcoparabolicity}, Proposition 4.3.5 and Corollary 5.1.3) Let $N\in \langle{M}\rangle$ be a constant overconvergent F-isocrystal. Then ${N^{\dagger}}\in \langle{M^{\dagger}}\rangle_L$.
\end{cor}

\begin{proof}
We show that the map $h$ in \Cref{lem:constant} is an isomorphism. 
Let $\langle N\rangle$ be the tensor generator of $\langle{M}\rangle^{\cst}$. As $N$ is constant, we see that $\langle N^{\dagger} \rangle\rightarrow\langle N\rangle$ is an equivalence of categories. We will show that $N^{\dagger}\in \langle M^{\dagger}\rangle$. Consider the Tannakian category $\langle M^{\dagger}\oplus N^{\dagger}\rangle$ and write $H^{\dagger}$ for its monodromy group. As $\langle M^{\dagger}\rangle\subseteq \langle M^{\dagger}\oplus N^{\dagger}\rangle$,  
  we get a surjection $j: H^{\dagger}\rightarrow G^{\dagger}$ between the respective monodromy groups. By construction, as $\langle M \oplus N\rangle = \langle M\rangle$ the map $j$ is an isomorphism restricted to $G$. Suppose $\ker j$ is not trivial. We have a surjective map $H^{\dagger}\rightarrow G^{\cst}.$ Note that $\Ker j$ does not intersect the kernel of this map as $\omega^{\dagger} (M^{\dagger}\oplus N^{\dagger})$ is a faithful representation of $H^{\dagger}$. Therefore, the action of the Newton cocharacter on $\ker j\subset G^{\cst}$ is trivial as $G^{\cst}$ is commutative. This shows that $\ker j\subset G$, but that is not possible. So it must be trivial. 
\end{proof}

The above result, together with Galois descent allows us to deduce the result we need:
\begin{cor}\label{cor:subquotientoverconverget}
Let $M$ be an overconvergent F-isocrystal over $\mathcal{E}$. Let $N \in \langle{M}\rangle$ be an overconvergent constant F-isocrystal. Then $N^{\dagger} \in \langle{M^\dagger}\rangle$.
\end{cor} 
\begin{proof}

Let $N$ be any constant overconvergent  isocrystal in $\langle{M}\rangle$. Write $N^{\dagger}$ for its exension to $\mathcal{E}^{\dagger}$. Then by \Cref{cor:constovercon}, $N^{\dagger}_L:=N^{\dagger}\otimes_{\Q_p}L\in \langle M^{\dagger}\rangle_L$. We now need to descend the scalers to $\Q_p$. Let $N^{\dagger}_L$ be a subquotient of $\oplus_i {M^{\dagger}}^{n_i}\otimes ({M^{\dagger}}^{\vee})^{m_i}=T(M^{\dagger})$. As $N^{\dagger}$ is defined over $\Q_p$, we see that $(N^{\dagger}_L)^{\sigma}\simeq N^{\dagger}$. Therefore, $(N^{\dagger}_L)^{\oplus m}$ is a subquotient of $\oplus_{\sigma\in \Gal(L/K)}T(M^{\dagger})^{\sigma}$ where $m=|\Gal(L/K)|$. By Galois descent, $(N^\dagger)^{\oplus m}\in \langle{M^{\dagger}}\rangle$ and therefore $N^{\dagger}\in \langle{M^{\dagger}}\rangle$.
\end{proof}

\section{Preliminary results on log F-isocrystals}

In this section, we prove some useful results, especially  \Cref{prop:unitroot_goodredn} that is key to proving both the finiteness results. 

\begin{rem}
In case the reader is first interested in reading the argument for \Cref{thm:mainthm}, they may skip directly to \Cref{prop:unitroot_goodredn}. All the results below leading to \Cref{prop:unitroot_goodredn} are extensions of the case when $\nabla_{\log}$ is regular, which are studied in \cite{ambrus_p_adic_monodromy} Section 2.
\end{rem}
\begin{lem}\label{lem:basishorizontal}
Let $(M, F, \nabla_{\log})$ be a log-F-isocrystal over $\textbf{F-Isoc}_{\log}^{\ddagger}(\o_K)$. Then there exists a basis of horizontal sections of $(M,F,\nabla_{\log})$ over $\mathcal{R}_+[\log u]$.
\end{lem}
\begin{proof}
We provide a sketch of the proof. (See also \cite{kedlaya_book}, Proposition 16.4.3). As in Proposition \cref{prop:redmodtlog}, by the logarithmic analog of Dwork's trick, we can find a basis $e_1,\dots, e_n$ of $M$ such that $\nabla_{\log} = d+N_0 du/u$ such that $N_0 = (N_{ij}(u))$ is an upper triangular nilpotent matrix with entries in $L$.  Let $e_1,\dots , e_r$ be such that $\nabla_{\log}(e_i) = 0$ for $1\leq i\leq r$. We define inductively a new horizontal basis of $M\otimes_{\mathcal{R}_+} \mathcal{R}_+[\log u]$ as follows. Let $e_i= w_i$ for $1\leq i\leq r$. For $i>r$, let $\nabla_{\log}(e_{i}) = \sum_{j=1}^{i-1} N_{ji}(u)e_{j}\otimes du/u$. For $i>r$, define $w_{i} := e_{i}-\sum_{j=1}^{i-1}H_{ji} e_{j}$ where $H_{ji}\in \mathcal{R}_+[\log u]$ are defined iteratively from the following relation: For a fixed $i$, and $1\leq j \leq i-1$, $d H_{ji}/du =(N_{ji}-N_{j,i-1}H_{j+1,i}-\dots - N_{j,1}H_{i-1,i})/u$. We note that for $k\in \Z$ and $m>0$, the expression  $u^k\log u^m$ has an antiderivative in $\mathcal{R}_+[\log u]$ and hence, $H_{ij}$ are well defined. We then have $\nabla_{\log}(w_{i}) =0$. This gives a horizontal basis of $M\otimes_{\mathcal{R}_+}\mathcal{R}_+[\log u]$ for the logarithmic connection $\nabla_{\log}$. 
\end{proof}
Let $(M_1,\nabla_{1}, F_1)$ and $(M_2, \nabla_{2}, F_2)$ be two log F-isocrystals over $R$. We can equip the set of $R$-linear maps $\Hom_R(M_1,M_2)$ with a structure of a log F-isocrystal by defining the connection as $(\nabla_{\log,\Hom}(f))(m) = \nabla_2(f(m))- (f\otimes \Id)\nabla_1(m)$. Since we're working with F-isocrystals, the Frobenii are invertible maps which helps us define the Frobenius $F_{\hom}(f):= F_2\cdot f\cdot F_1^{-1}$. We write $\underline{\Hom}(M_1,N_1)$ to denote this log-$F$-isocrystal.

 \begin{lem}\label{lem:homeq}
Let $M_1$ and $M_2$ be log isocrystals over $R$ as above. Then there is an isomorphism between the $L$-vector spaces 
$$\Hom_{F,\nabla_{\log}}(M_1,M_2)\simeq H^0(\underline{\Hom}(M_1,M_2)):=(\underline{\Hom}(M_1,M_2))^{F_{\Hom}=1,\nabla_{\log,\Hom}=0}$$
 \end{lem}
 \begin{proof}
Let $f$ be an element of $\Hom_{F,\nabla_{\log}}(M_1,M_2)$. Then commutativity with the Frobenii of $M_1$ and $M_2$ is the same as $F_2\cdot f= f\cdot F_1$, that is $F_2\cdot f \cdot F_1^{-1}=f$. This is exactly the condition that $F_{\Hom}(f)=f$. The commutativity with the connection implies that $\nabla_2 \cdot f = (f\otimes \Id)\cdot \nabla_1$ which is exactly the condition that $\nabla_{\log,\Hom}=0$.
 \end{proof}
 We use \Cref{lem:basishorizontal} and \Cref{lem:homeq} to prove the following proposition, which is the local version of Theorem 7.3 in \cite{kedlaya_notes}. When $\nabla_{\log}$ is regular, the argument is the same and works with $\mathcal{R}_+$ itself in place of $\mathcal{R}[\log u]$ using the regular Dwork's trick as in the proof of \cite{kedlaya_book}, Theorem 8.7.1.

\begin{prop}\label{prop:extfullyfaithful}
Let $M_1$ and $M_2$ be two log F-isocrystals over $\mathcal{E}_+$ in $\textbf{F-Isoc}_{\log}(S)$. Let $M_1^{\dagger}$ and $M_2^{\dagger}$ denote their extension to $\mathcal{E}^{\dagger}$. Then the forgetful map: $$\Hom_{F,\nabla_{\log}} (M_1,M_2)\rightarrow \Hom_{F,\nabla}(M_1^{\dagger}, M_2^{\dagger})$$ is an isomorphism.
\end{prop}
\begin{proof}
The map is injective under extension of scalers. To  show surjectivity, we consider the log $F$-isocrystal $M':=\underline{\Hom}(M_1, M_2)$ over $\mathcal{E}_+$. By Lemma \ref{lem:homeq}, we need to show any horizontal section of $M'\otimes_{\mathcal{E}_+} \mathcal{E}^{\dagger}$ lifts to a horizontal section of $M'$. By \Cref{lem:basishorizontal}, we can find a bases $e_1,\dots ,e_r$ of horizontal vectors of $M'\otimes_{\mathcal{E}_+}\mathcal{R}_+[\log u]$. As $\mathcal{R}_+\subset \mathcal{R}$, the same set of horizontal vectors trivialise $M'\otimes_{\mathcal{E}_+}\mathcal{R}[\log u]$. Let $s$ be a horizontal section of $M'\otimes_{\mathcal{E}_+} \mathcal{E}^{\dagger}.$ Then $s$ viewed as a horizontal section of $M'\otimes_{\mathcal{E}_+}\mathcal{R}[\log u]$ can written as $c_1e_1+\dots +c_re_r$. Since $s$ is trivial, we see that the $c_i$ have to be constant for all $1\leq i\leq r$. Hence, $s\in \mathcal{R}_+[\log u]$ which implies that $s\in \mathcal{R}_+[\log u]\cap \mathcal{E}^{\dagger}=\mathcal{E}_+$. 
\end{proof}
We will also need the following results.

\begin{lem}\label{lem:nilbasis}
Let $A^{\dagger}\subseteq M^{\dagger}$ over $\mathcal{E}^{\dagger}$ be a subobject of $M^{\dagger}$. Suppose $M^{\dagger}\otimes_{\mathcal{E^{\dagger}}}\mathcal{R}$ has a basis such that the connection $\nabla_{\log} = d+N_0 du/u$ with $N_0$ nilpotent with entries in $L^{\ur}$ then so does $A^{\dagger}\otimes_{\mathcal{E}^\dagger}\mathcal{R}$. 
\end{lem}
\begin{proof}
We use \cite{ambrus_p_adic_monodromy}, Lemma 2.12 which is the regular version of this result inductively. Consider the filtration $M_0\subseteq M_1\subseteq\dots \subseteq M^{\dagger}\otimes_{\mathcal{E}^{\dagger}} \mathcal{R}=M_n$ induced by $N_0$ such that $M_i/M_{i-1}$ is trivial. Intersect $A\otimes_{\mathcal{E}^\dagger}\mathcal{R}$ with the above filtration to get a filtration of $A$ such that $A_{i}/A_{i-1}\hookrightarrow M_i/M_{i-1}$. As $M_i/M_{i-1}$ is trivial  \cite{ambrus_p_adic_monodromy}, Lemma 2.12 to find a basis of horizontal vectors of $\nabla_{\log}\otimes \mathcal{R}|_{A\otimes_{\mathcal{E}^\dagger}\mathcal{R}}$ for $A_i/A_{i+1}$ for all $i\geq 1$. Lifting these to a basis of $A$. By construction, $\nabla_{\log}=d+Nd/dt$ where $N\in \mathcal{R}$ is nilpotent. As $A_0$ is trivial, we know that the entries $N|_{A_0}$ lie in $L^{\ur}$. Assume $N|_{A_i}$ also has entries in $L^{\ur}$ for $i>1$. Let $e_1, \cdots e_r, f_1,\cdots f_r$ denote the basis for $A_{i+1}$ such that $e_i$ span $A_i$ and $f_i$ lift the horizontal basis of $A_{i+1}/A_i$. Let $f_j = \sum a_{ij}(u)e_i du/u$ with $a_{ij}(u)\in \mathcal{R}$. Replace each $f_j$ by $f_j - \sum_i b_{ij}(u)e_i$ with $b_{ij}(u)\in \mathcal{R}$ where we find $b_{ij}(u)$ by solving $\nabla(f_j) = \sum_i a_{i,j}(0)e_i du/u$. That is, for each $i$, we have $u d (b_{ij}(u))/du + N|_{A_i} b_{ij}(u) =a_{ij}(u) -a_{ij}(0).$ Writing $b_{ij}$ and $a_{ij}$ as power series and comparing coefficients for $u^n$, we find a solution. The lemma now follows by induction. 
\end{proof}
\begin{prop}\label{prop:pal}
Let $M^{\dagger}$ be an F-isocrystal over $\mathcal{E}^{\dagger}$ such that $M^{\dagger}\otimes_{\mathcal{E}^{\dagger}} \mathcal{R}\simeq M'\otimes_{\mathcal{R}_+} \mathcal{R}$ where $M'$ is a log F-isocrystal over $\mathcal{R}_+$. Then there exists a log F-isocrystal $M$ over $\mathcal{E}_+$ such that $M\otimes_{\mathcal{E}_+}\mathcal{E}^{\dagger}\simeq M^{\dagger}$. 
\end{prop}
\begin{proof}
This is Theorem 2.10 in \cite{ambrus_p_adic_monodromy} when $M'$ is an F-isocystal over $\mathcal{E}_+$. The proof in the logarithmic case is verbatim the same with $ u\cdot d/du $ in place of $d/du.$
\end{proof}
\begin{prop}\label{prop:pal2}
Let $A^{\dagger}$ be an F-isocrystal over $\mathcal{E}^{\dagger}$ which is a sub F-isocrystal
of $M\otimes_{\mathcal{E}_+} \mathcal{E}^{\dagger}$, where $M$ is a log F-isocrystal over
$\mathcal{E}_+$. Then there is a sub-log F-isocrystal $A$ of $M$ over $\mathcal{E}_+$ such that $A^{\dagger}\simeq A\otimes_{\mathcal{E}_+} \mathcal{E}^{\dagger}.$ 
\end{prop}
\begin{proof}
 When $M$ is an F-isocrystal over $\mathcal{E}_+$, this is Theorem 2.14 in \cite{ambrus_p_adic_monodromy}. By \Cref{prop:redmodtlog}, $M\otimes_{\mathcal{E}_+}\mathcal{R}$ admits a basis such that $\nabla_{\log} = d+N_0 du/u$ such that the entries of $N_0\in L$. The same basis work over $\mathcal{M}\otimes_{\mathcal{E}_+}\mathcal{R}$. By \Cref{lem:nilbasis}, there exists a basis for $A^\dagger\otimes_{\mathcal{E}^{\dagger}}\mathcal{R}$ such that $\nabla_{\log}|_{A^\dagger} = d+N_0' du/u$ where $N_0'$ is nilpotent with entries in $L$. Consequently, there exists a log isocrystal $A'$ over $R_+$ such that $A'\otimes_{\mathcal{R}_+}\mathcal{R}\simeq A$. By \Cref{prop:pal}, there exists a log F-isocrystal $A$ over $\mathcal{E}_+$ such that $A\otimes_{\mathcal{E}_+}\mathcal{E}^\dagger\simeq A^\dagger$. By \Cref{prop:extfullyfaithful}, $A\subset M$.
\end{proof}
\begin{notn}
Let $M$ be a log F-isocrystal over $\mathcal{E}_+$. We use the notation $\langle{M}\rangle$ for the smallest tensor category generated by $M$ in $\textbf{F-Isoc}_{\log}^{\dagger}(\o_K)$ closed under subobjects, duals, and quotients whenever the latter exist in the ambient category. When $\nabla_{\log}$ is regular, then $\langle{M}\rangle$ will be understood as in \Cref{notn:tannakian}.
\end{notn}

Let $\overline{M}\in \textbf{F-Isoc}_{\log}(\mathcal{O}_K)$ be a log F-isocrystal. Write $M:= \overline{M}\otimes\mathcal{E}$. Let $U$ be a unit root subobject of $M$. Denote the Galois group $\Gal(K^{\sep}/K)$ by $G_K$ and its inertia subgroup by $I_K$. Let $G\subset \GL(V)$ denote the Zariski closure of the image of $G_K$ in the associated Galois representation and $H$ denote the Zariski closure of $I_K$.
\begin{prop}\label{prop:unitroot_goodredn}
We keep the above notation. Let $\mathbb{W}$ be any representation of $G/H$ and $W$ over $\mathcal{E}$ be the associated unit root F-isocrystal. Then there exists an F-isocrystal $\overline{W}$ in $\langle{\overline{M}}\rangle$ over $\mathcal{E}_+$ such that $\overline{W}\otimes \mathcal{E} = W$ and the pullback of $\overline{W}$ to $L^{\ur}$, yields a unit root Dieudonn\'e module $W_0$ in $\langle{\overline{M}\otimes L^{\ur}}\rangle$.
\end{prop}
\begin{proof}
As $\mathbb{W}$ is unramified, by \Cref{thm:tsuzuki_unit_root}, $W$ is overconvergent. Therefore, there exists an F-isocrystal $W^{\dagger}$ over $\mathcal{E}^{\dagger}$ such that $W^{\dagger}\otimes \mathcal{E} =W$. As $V$ is a faithful representation of $G$, by \Cref{thm:tsuzuki_unit_root} and the proof of Proposition 2.20 (b) in \cite{Milne_Deligne_tannakian}, we see that $\langle{U}\rangle \simeq \Rep (G)$. This implies $\mathbb{W}$ viewed as a representation of $G$ is an object in $\langle{V}\rangle$.  By  \Cref{thm:tsuzuki_unit_root}, $W\in \langle U\rangle\subset \langle M\rangle$. By construction $W$ is constant and hence \Cref{cor:subquotientoverconverget} implies that $W^{\dagger}\in \langle{M^{\dagger}}\rangle$. Now let $A^{\dagger}\subseteq \oplus_i{M^{\dagger}}^{\otimes n_i}\otimes ({M^\dagger}^{\vee})^{\otimes m_i}:= T(M)$ be such that $W^{\dagger}$ is a quotient of $A^{\dagger}$. By \Cref{prop:pal2}, there exists a sub log-F-isocrystal $A\subset M$ over $\mathcal{E}_+$ such that $A^{\dagger}\simeq A\otimes_{\mathcal{E}_+} \mathcal{E}^{\dagger}.$
We further note that $W$ is the pull back of a unit root F-isocrystal over $L$. By \Cref{prop:redmodtlog}, we can find F-isocrystal $W'$ over $\mathcal{R}_+$ such that $W'\otimes_{\mathcal{R}_+}\mathcal{R}=W^{\dagger}\otimes_{\mathcal{E}^{\dagger}} \mathcal{R}$.  By \Cref{prop:pal}, we see that there exists an F-isocrystal $\overline{W}$ over $\mathcal{E}_+$ such that $\overline{W}\otimes_{\mathcal{E}_+}\mathcal{E}^{\dagger}\simeq W^{\dagger}$. By applying \Cref{prop:extfullyfaithful} to $A^{\dagger}\rightarrow W^{\dagger}$ we see that $\overline{W}\in \langle \overline{M}\rangle$. Reducing mod $t$ finishes the proof of the lemma. 
\end{proof}

\section{The case of good reduction}\label{sec:proofgoodred}

 We now prove \Cref{thm:mainthm}. 

\begin{defn} 
Let $M_0$ be an F-isocrsytal over $L^{\ur}$ and let $\varphi$ denote the linearised Frobenius $F^s$ where $s=[
\F_q:\F_p].$ We will call $M_0$ \textit{decent} if the action of $\varphi$ on $M_0$ is semi-simple, algebraic, $p$-plain, and isoclinic.
\end{defn}
\begin{rem}\label{rem:decent}
This is a mild generalisation of the termininology of \textit{decent}  given in \cite{KLSS}, Section 2.8. This allows us to write the proof in a more general set-up as described in \Cref{mainresult}. In particular, the above definition matches with them when $M$ is the Dieudonn\'e module of a supersingular abelian variety over $\F_q$, all the eigen-values $\varphi$ are equal and are rational power of $q$. 
\end{rem}

Recall the set-up of \Cref{thm:mainthm}. Let $\mathcal{M}^{\dagger}$ be an overconvergent F-isocrystal over $X$ over $\F_p$ such that the action of the Frobenius is algebraic, $p$-plain and semisimple at closed points. Let $Z\subset X$ be such that $\mathcal{M}^{\dagger}|_{Z}$ has a non-trivial unit root sub-object $\mathbb{U}$. Write $\eta$ to be the generic point of $Z$ and $x$ for its specialisation in the interior of $\overline{Z}$. By pulling back $\mathbb{U}$ to $\eta$  in a local neighborhood of $x$, we get a Galois representation:
$$\rho_x: \Gal(K^{\sep}/K)\rightarrow \GL_n(\Q_p).$$
Let $G$ and $H$ denote the Zariski closure of the image of $G_K$ and $I_K$ in $\GL_n(\Q_p)$ respectively. We note that $[\rho_A(G_K):\rho_A(I_K)]$ is finite if and only if $[G:H]$ is finite (\cite{KLSS}, Lemma 2.3).
 Let $M^{\dagger}$ over $\mathcal{E}^{\dagger}$ denote the pullback of $\mathcal{M}^{\dagger}$ to $\eta$ in a local neighborhood of $x$. Let $\overline{M}$ be its extension to $\mathcal{E}_+$. Write $M = \overline{M}\otimes \mathcal{E}$ and let $U$ be the unit root sub-F isocrystal of $M$ associated to $\rho_x$. Write $M_0$ for the Dieudonn\'e module of $x$, that is, the base change of $\overline{M}\otimes_{\mathcal{E}_+} L^{\ur}$.
\begin{thm}[\Cref{thm:mainthm}]\label{proofmainthm}
Suppose $x$ is isoclinic, then $[G:H]$ is finite.
\end{thm}

\begin{proof}
Let $\mathbb{W}$ be a faithful representation of $G/H$.  By applying \Cref{prop:unitroot_goodredn} to $W$, we get an F-isocrystal $\overline{W}\in \langle{\overline{M}}\rangle$. Therefore, $\overline{W}\otimes L^{\ur}\in \langle M_0\rangle $. As $x$ is isoclinic and $\mathcal{M}^{\dagger}$ is algebraic and $p$-plain at $x$, we see that $M_0$ is decent. Therefore, the action of the (linearised) Frobenius on $W_0$ is semisimple and hence, all the eigen-values of the Frobenius are of the form $\alpha = \alpha_1^{d_1}\cdots \alpha_n^{d_n}$ where $\alpha_i$ are the eigen values of the Frobenius on $M_0$. Furthermore, we know that  $\alpha_i$ (and hence $\alpha$) are algebraic and that $v(\alpha_i) = v(\alpha_j) = r $ for all places lying over $p$. Choose a place $\mathfrak{p}$ lying over $p$ and fix an embedding $\Q_p\hookrightarrow \overline{\Q_p}$. Since $W_0$ is unit root, the eigenvalues of the (linearised) Frobenius on $W_0$ are $\mathfrak{p}$-adic units. Therefore, the valuation $v_{\mathfrak{p}}(\alpha) = 0\sum_i d_i v_{\mathfrak{p}}(\alpha_i)$. We normalise the valuations by $v_{\mathfrak{p}}(q)$ where $q$ is such that $\F_q$ is the residue field of $x$. As the Newton polygon is isoclinic, we write write $v_\mathfrak{p}(\alpha_i)/v_{\mathfrak{p}}(q)=r$. This implies that $r=0$ or $\sum_i d_i=0$. Now choose any other place $\mathfrak{p}'\mid p$. Then $v_{\mathfrak{p}'}(\alpha_i)/v_{\mathfrak{p}'}(q)=r$. Therefore, $v_\mathfrak{p'}(\alpha)/v_\mathfrak{p'}(q) = r\sum_i d_i=0$. Therefore, we see that $v(\alpha)=0$ at all places $v$ lying above $p$. As $\alpha_i$ are units away from $p$ as well, we see that $\alpha$ is a root of unity and the action of the Frobenius on ${W}_0$ has finite order. Hence, for some finite extension $L^{\prime}$ of $L^{\ur}$ we see that the Galois represenation can identified with $({U}_0\otimes_{L^{\ur}} L^{\prime})^{\varphi=1}$ with $G_K/I_K$ acting on $L'$. Thus the action of $G_K/I_K$ factors through a finite quotient. As $\mathbb{W}$ is a faithful representation of $G/H$, we see that $[G:H]$ is finite.
\end{proof}

The application of \Cref{thm:mainthm} to Shimura varieties is immediate. We therefore refer the reader to \Cref{cor:goodredSV} to recall the Notation.
\begin{cor}[\Cref{cor:goodredSV}]
Let $x$ be a point $\mathscr{S}_{k(v)}(\F_q)$ with isoclinic Newton slopes. Assume \ref{assumtionfrobss} holds in general. Then the image of inertia is finite index in the image of $\rho_x$.
\end{cor}

\begin{proof}
 \cite{Huryn-Patrikis-Klevdal-Kedlaya} and \cite{Patrikis} (Theorem 1.1 and Section 4), the overconvergent F-isocrystal $\mathcal{M}^{\dagger}$ over $\mathscr{S}_{k(v)}$ is algebraic and $p$-plain at closed points of $\mathscr{S}_{k(v)}$. As noted in the introduction, Frobenius semisimplicity holds for Shimura varieties of abelian type. For exceptional Shimura varieties, we assume \ref{assumtionfrobss} to hold. Furthermore, as $x$ belongs to the basic stratum of $\mathscr{S}$, and $V$ is an irreducible representation of $G_{\Q_p}$, by \cite{kottwitz} Proposition 1.12, we see that $\mathcal{M}^{\dagger}_x$ is isoclinic.  The corollary now follows from \Cref{thm:mainthm}. 
\end{proof}

\section{Local monodromy of abelian varieties with semi-stable reduction}

\subsection{Interlude on Raynaud extensions} \label{raynaud}
In this section we review uniformation of abelian varieties with semi-stable reduction. To every abelian variety with semi-stable reduction over $K$, we can associate a log F-isocrystal via the theory of $1$-motives and Raynaud extensions. We briefly describe this below along with the data we will need to prove \Cref{thm:badredabvar}. 
\begin{defn}(See \cite{kato_trehan} 4.5.1)\label{logmotive}
We define a 1-motive over $K$ to be a triple $(\Lambda_K, Z,f)$ such that $\Lambda$ and $Z$ are commutative group schemes over $K$ satisfying
\begin{enumerate} 
\item The group scheme $Z$ is semi-abelian.
\item  The group scheme $\Lambda_K \simeq \Z^r$ \'etale locally on $\Spec K$.
\item $f$ is a homomorphism $\Lambda_K\rightarrow Z$ over $K$ .
\end{enumerate}
\end{defn}
The same data over $\o_K$ defines a log-1 motive:
\begin{defn}(See \cite{kato_trehan} 4.6.1)\label{logmotive}
We define a log-1-motive over $\o_K$ to be a triple $(\Lambda, \mathcal{Z},f)$ such that $\Lambda$ and $\mathcal{Z}$ are commutative group schemes over $\o_K$ satisfying $(1)$ and $(2)$ above over $\o_K$ and $f: \Lambda_K\rightarrow\mathcal{Z}_K$ is a group homomorphism over $K$.
\end{defn}

The data of a 1-motive $(\Lambda_K,Z,f)$ such that $T$ and $B$ denote the toric subscheme and the abelian quotient of $Z$. Denote by $X^{*}(T)$ equals the group of characters of $T$. Then this data is equivalent to the triple $(h_K,h_K^*,s_K)$ defined as follows: Let $h_K: \Lambda_K\rightarrow A$ is a group homomorphism, $h_K^*: X^{*}(T)\rightarrow A^{\vee}$ along with a trivialisation $s_K: Y\times X^{*}(T)\rightarrow \mathcal{P}_K$ of the pullback along $h_K\times h_K^*$ of the Poincar\'e biextension $\mathcal{P}_K$ on $B\times B^{\vee}$. Since $B$ has good reduction, $\mathcal{P}_K$ extends to a biextension on $\mathcal{B}\times \mathcal{B}^{\vee}$ where $\mathcal{B}$ over $\o_K$ is the N\'eron model of $B$. The morphisms $h_K$ and $h_K^*$ extend to maps $h$ and $h^{*}$ on $\o_K$ as well. The pullback of $\mathcal{P}$ along $h\times h^{*}$ gives a biextension which is trivial generically because of $s_K$. This yields the existence of a \textit {geometric monodromy pairing}:
$\nu: \Lambda_K\times X^*(T)\rightarrow \Z$ that gives a morphism $\nu: \Lambda_K\rightarrow \Hom(X^*(T),\Z)=X_{*}(T)$. Assuming that $T$ is split, we get a morphism of tori:
$\nu\otimes 1: \Lambda\otimes \mathbb{G}_m\rightarrow X_{*}(T)\otimes \mathbb{G}_{m}=T$ that extends to $\o_K$ (see \cite{monodromylogiso}, Section 1.2).

We now relate this to the theory of Raynaud extensions coming from degenerating abelian varieties. Let $A$ is an abelian variety with semi-stable reduction over $K$. The special fiber of its N\'eron model of $A$ is a semi-abelian variety which we denote by $A_0$. Completing the identity component of the N\'eron model along $A_0$, we get a formal scheme $\widehat{Z}$ over $\o_K$ such that $\widehat{Z}$ is an extension of a formal torus $\widehat{T}$ by a formal abelian scheme $\widehat{B}$ with $\Lambda = \Hom(\widehat{T}^{\vee}, \widehat {\mathbb{G}_m})$. The extension $\widehat{Z}$ can be algebraized to give an extension $\mathcal{Z}$ of an abelian variety $\mathcal{B}$ by a torus $\mathcal{T}$ over $\o_K$ and $\Lambda =\Hom (\mathcal{T}^{\vee}, \mathbb{G}_m)$, see \cite{chai_faltings}, Chapter II Section 1. The ``Raynaud generic fiber" gives an isomoprhism of rigid analytic varieties $A(\overline{K}) = Z(\overline{K})/\Lambda$ where
$\Lambda$ is realized as a multiplicative lattice in $Z(\overline{K})$.  We call $Z$ the \textit{Raynaud extension} of $A$ (see \cite{chai_faltings} Theorem 6.2). In summary, we have the following diagram over $K$: 

$$\xymatrix{ &\Lambda \ar[d]^{f}&\\ T\ar[r] & Z \ar[r] \ar[d]&B& \\
& A&  }$$
 As $T$ extends over $\o_K$, it is unramified and it splits over a finite unramified extension of $K$ and therefore, up to a finite extension, we can assume that the lattice $\Lambda$ is split. 

To every one motive, we can associate a $p$-divisble group as describe in \cite{kato_trehan} 4.5.3. Further, via composing with the co-variant Dieudonn\'e functor $\mathbb{D}(\cdot)$, we get an associated Dieudonn\'e crystal on the crystalline site of $\Spec K$ (see \cite{kato_trehan}, 4.5.4). Via the $p$-adic uniformisation theory, the $p$-divisible group associated to the 1-motive of $A$ identifies with the $p$-divisible group of $A$. There is also a log-Dieudonn\'e functor obtained as a composition of functors from log-1-motives to log $p$-divisible groups and then to log-Dieudonne crystals. In \cite{kato_trehan} 4.7, the authors give a direct construction of the log Dieudonn\'e functor $\mathbb{D}_{\log}(\cdot)$ from log-1 motives to Dieudonn\'e crystals on the log scheme $\Spec \o_K$ with the log structure associated to the closed point $t=0$, such that its pullback to $\Spec K$ is compatible with the Dieudonn\'e functor. The  pullback of a log-1 motive $(\Lambda,\mathcal{Z}, f):= \mathbb{M}$ over $\o_K$ to $\Spec K$ will be understood as the $1$-motive  $(\Lambda_K,Z, f)$.

To maintain consistency with the previous section, we will work with the contravariant (log) Dieudonn\'e functor which is obtained by taking the dual of the log Dieudonn\'e crystal as in \cite{kato_trehan}, Section 4.1.
\subsection{Filtration on the Dieudonn\'e realization of 1-motives}  Every log 1-motive $\mathbb{M}=(\Lambda,\mathcal{Z}, f)$ inducues a natural 3-step weight filtration on the Dieudonn\'e crystal $W_{\bullet}\mathbb{D}_{\log}(\mathbb{M})$ given by: 
$$0\subset \Lambda\otimes \Q_p/\Z_p\subset \mathbb{D}_{\log}(\Lambda,\mathcal{B},\overline{f})\subset \mathbb{D}_{\log}(\mathbb{M})$$
where $\bar f: \Lambda\rightarrow \mathcal{B}$ is the composition of $f$ and the quotient map from $\mathcal{Z}$ to $\mathcal{B}$. Their pullback to $\Spec K$ gives the canonical filtration on the Dieudonn\'e crystal associated to the 1-motive $(\Lambda_K, Z, f)$. We have the following short exact sequences from these filtration.
\begin{enumerate}

 \item \label{1} Then we have the following exact sequence: $$0\rightarrow \mathbb{D}_{\log}(\Lambda, \mathcal{B}, \bar f)\rightarrow \mathbb{D}_{\log}(\mathbb{M}) \rightarrow \mathbb{D}_{\log}(\mathcal{T})\rightarrow  0$$

    \item \label{2}The $2$-step filtration gives us the following exact sequence of log Dieudonn\'e modules: $$0\rightarrow \Lambda\otimes \Q_p/\Z_p\rightarrow  \mathbb{D}_{\log}(\mathbb{M})\rightarrow \mathbb{D}_{\log}(\mathcal{Z})\rightarrow 0$$ 
\end{enumerate}
We evaluate the Dieudonn\'e crystal on the the logarithmic PD-thickenings $\Spec \o_K\hookrightarrow W_n(\F_q)[[t]]$ with the canonical log structure associated to $t=0$, take the inverse limit over $n$ and invert $p$, to work with a log F-isocrystal over $\mathcal{E}_+$ as defined in \Cref{2.3}. We keep the same notation of the previous paragraph for the associated log F-isocrystal over $\mathcal{E}_+$ and their pullbacks to $\mathcal{E}$.

By the the discussion in Section \ref{2.3}, pulling back $\mathbb{D}_{\log}(M)$ to the divisor $t=0$ will give us a log Dieduonn\'e module $(M_0,F_0, N_0):=\mathbb{D}_{\log}(\Lambda, \mathcal{Z},f)\otimes L^{\ur}$ over $L^{\ur}.$ The geometric monodromy pairing $\nu$ induces the nilpotent operator
$N_0: M_0\rightarrow M_0$ which is given by the composition: $$M_0\rightarrow \mathcal{D}_{\log}(\mathcal{T})\otimes L^{\ur}\xrightarrow{(\nu\otimes 1)\otimes L^{\ur}} (\Lambda\otimes {\mu_{p^{\infty}}})\otimes L^{\ur}\xrightarrow{\Id(1)} (\Lambda\otimes \Q_p/\Z_p)\otimes L^{\ur}\rightarrow M_0$$ where $\Id(1)$ is a $W(k)(1/p)$-linear morphism that takes the generator of $\mu_{p^{\infty}}$ to that of $\Q_p/\Z_p$ and satisfies the condition $F\circ \Id(1)\circ V=\Id(1)$.  Therefore, the image of $N_0$ lies in ${\Lambda}\otimes\ \Q_p/\Z_p\hookrightarrow M_0$ and we have $N_0^2=0$ (see for instance, \cite{monodromylogiso}, Section 3). 

\begin{lem}\label{lem:smalllem2}
Consider the Dieudonn\'e module $\mathbb{D}_{\log}(\Lambda, \mathcal{B}, \overline{f})\otimes_{\mathcal{E}_+} L^{\ur}:=(W_0, N_0|_{W_0}, F_0|_{W_0})$  over $L^{\ur}$. Then $N|_{W_0}=0$ and we can find a basis of $W_0$ such that the (linearised) Frobenius
$$
F = \begin{pmatrix}
    \Id& 0\\
    0& F_{B_0}\\
\end{pmatrix} 
$$
where $F_{B_0}$ is the (linearised) Frobenius of the Dieudonn\'e module of $B_0$. 
\end{lem}
\begin{proof}

By the above description of $N_0$ and the exact sequence in (\ref{2}), we see that $N_0|_{W_0}=0$. By extending the basis of $\Lambda\otimes \Q_p/\Z_p$ to $W_0$, we write the Frobenius as an upper block form with the diagonal matrices $\Id$ and $F_{B_0}$. As $B_0$ is a supersingular abelian variety, $F_{B_0}$ is diagonalisable and has eigen values different than $1$. Therefore, we can put $F$ in the above diagonal form.
\end{proof}

\subsection{Proof of \Cref{thm:badredabvar}} Suppose $A/K$ has semi-abelian reduction. We use the same notation from Section \ref{raynaud} for  the data of the Raynaud extension. Let $\mathbb{U}$ be the unit root F-isocrystal associated to the $p$-power torsion points of $A$ and write $\rho_A$ to be the associated Galois representation. Let $G$ be the the Zariski closure of image of $\rho_A$ and $H$ the Zariski closure of the inertia subgroup.

\begin{thm}(\Cref{thm:badredabvar})
Suppose $A$ has semi-supersingular reduction. Then $G/H$ is finite.
\end{thm}
\begin{proof}
As $\mathbb{U}\subseteq \mathbb{D}(\Lambda_K, Z, f)$ is unit root and $\mathbb{D}(T)$ has slope one, the composition $\mathbb{U}\subset \mathbb{D}(\Lambda_K, G, f)\rightarrow \mathbb{D}(T)$ is zero. By (\ref{2}), $\mathbb{D}(\Lambda_K, B, f)$. Let $\mathbb{W}$ be a faithful representation of $G/H$.  By Theorem \ref{thm:tsuzuki_unit_root}, we can associate an overconvergent unit root $F$-isocrystal $W\in \langle \mathbb{D}(\Lambda, B,\overline{f})\rangle$ over $\mathcal{E}$.  By \Cref{prop:unitroot_goodredn}, there exists an F-isocrystal $\overline{W}$ over $\mathcal{E}_+$ such that  $\overline{W}\in \langle{\mathbb{D}_{\log}(\Lambda,\mathcal{B},f)}\rangle$ and $\overline{W}\otimes L^{\ur} \in \langle W_0\rangle$ where $W_0:= \mathbb{D}_{\log}(\Lambda, \mathcal{B},\bar f)\otimes L^{\ur}$.  As the (linearised) Frobenius of $B_0$ is semisimple, by Lemma \ref{lem:smalllem2} the (linearised) Frobenius on $W_0$ is semisimple.
Further as $B_0$ is supersingular, $\overline{W}\otimes L^{\ur}$ is decent. Hence, all the eigen values of Frobenius on $\overline{W}\otimes L^{\ur}$ are of the form $\alpha = \alpha_1^{d_1}\cdots\alpha_n^{d_n}$ where $\alpha_i$ are eigen values of $F_{B_0}$. Therefore, we can now run the proof of \Cref{proofmainthm} to conclude that $G/H$ is finite.
\end{proof}

\section{Reduction of Hecke orbit}The proof of the following corollary follows analogously to Corollary 2.10 in \cite{KLSS} and Corollary 8.2 in \cite{myownpaper}. For the sake of completeness, we present it here in the equi-characteristic setting.
\begin{cor}
Let $A/K$ be an ordinary abelian variety. Suppose that $A$ has supersingular reduction  (resp. semi-supersingular reduction). Then the reduction of its Hecke orbit (resp. $p$-power Hecke orbit) is finite.
\end{cor}
\begin{proof}
 In both the cases, by \Cref{cor:goodredSV} and \Cref{thm:badredabvar}, it is enough to replace $K$ by a finite extension and assume that the $p$-power torsion of $A$ is defined over a ramified extension of $K$. Consequently, all the $p$-power isogenies are defined over a ramified extension of $K$. Hence the reduction of the $p$-power Hecke orbit of $A$ is defined over a fixed finite field extension $\F_{q'}$ of $k$. As there are finitely many isomorphism classes of semi-abelian varieties over a fixed finite field, the result follows when $A$ has semi-supersingular reduction. 

On the other hand, more is true when $A$ has good supersingular reduction. As the prime-to-$p$ power torsion of $A$ is entirely determined by the action of the Frobenius on its reduction over $k$, we can assume that the action of the Frobenius on the prime-to-$p$ power torsion is via scaling by $q^{1/2}$, up to a finite extension of the residue field of $K$. By \Cref{cor:goodredSV}, replacing $K$ with a finite extension, the $p$-power isogenies are defined over a ramified extension of $K$. Their reduction is therefore defined over a fixed finite field. The result now follows as there are only finitely many isomorphism classes of abelian varieties over a finite field. 
\end{proof}

 \bibliography{bibli}
\bibliographystyle{amsalpha}
 \end{document}